\documentclass{amsart}
\usepackage{amsfonts,amssymb,amsthm,amsmath,euscript}
\usepackage{colonequals}
\usepackage{units}
\usepackage{subfigure}
\usepackage{graphicx}
\usepackage{tikz}
\usepackage{pinlabel}
\usepackage[colorlinks,citecolor=blue,linkcolor=red]{hyperref}
\input{xy}
\xyoption{all}

\newcommand{\Out}{\mathrm{Out}}

\newcommand{\FN}{F_N}

\newcommand{\inv}{^{-1}}

\newcommand{\R}{{\mathbb R}}

\newcommand{\Z}{{\mathbb Z}}

\newcommand{\Hom}{\mathrm {Hom}}

\newcommand{\flow}{\psi}   

\theoremstyle{plain}
\newtheorem{theorem}{Theorem}[section]
\newtheorem{lemma}[theorem]{Lemma}
\newtheorem{corollary}[theorem]{Corollary}

\newtheorem{proposition}[theorem]{Proposition}
\newtheorem{prop-defn}[theorem]{Proposition-Definition}

\newtheorem{claim}[theorem]{Claim}
 
\newtheorem*{theorem:Lipschitzflows}{Theorem \ref{T:lipschtiz flows}} 
\newtheorem*{theorem:conesections}{Theorem \ref{T:cone_of_sections}}
\newtheorem*{theorem:splittings}{Theorem \ref{T:splittings}}
\newtheorem*{theorem:determinant}{Theorem \ref{T:determinant formula}}
\newtheorem*{theorem:continuity-convexity}{Theorem \ref{T:continuity/convexity again}}
\newtheorem*{theorem:cones-are-equal}{Theorem \ref{T:cones_are_equal}}
\newtheorem*{theorem:transversefoliations}{Theorem \ref{T:transverse foliations}}
\theoremstyle{definition}
\newtheorem{defn}[theorem]{Definition}

\newtheorem{question}[theorem]{Question}
\newtheorem*{remark*}{Remark}

\newtheorem{example}[theorem]{Example}

\makeatletter
\let\c@equation\c@theorem
\makeatother
\numberwithin{equation}{section}

\begin{document}

\title[Endomorphisms and train track maps]{Endomorphisms, train track maps, and fully irreducible monodromies}
\author{Spencer Dowdall, Ilya Kapovich, and Christopher J. Leininger}

\address{\tt  Department of Mathematics, Vanderbilt University,
  1326 Stevenson Center, Nashville, TN 37240
 \newline \indent http://www.math.vanderbilt.edu/\~{}dowdalsd/} \email{\tt spencer.dowdall@vanderbilt.edu}

\address{\tt  Department of Mathematics, University of Illinois at Urbana-Champaign,
  1409 West Green Street, Urbana, IL 61801
  \newline \indent http://www.math.uiuc.edu/\~{}kapovich/} \email{\tt kapovich@math.uiuc.edu}

\address{\tt  Department of Mathematics, University of Illinois at Urbana-Champaign,
  1409 West Green Street, Urbana, IL 61801
  \newline \indent http://www.math.uiuc.edu/\~{}clein/} \email{\tt clein@math.uiuc.edu}
  
\begin{abstract}
Any endomorphism of a finitely generated free group naturally descends to an injective endomorphism of its stable quotient. 
In this paper, we prove a geometric incarnation of this phenomenon: namely, that every expanding irreducible train track map inducing an endomorphism of the fundamental group gives rise to an expanding irreducible train track representative of the injective endomorphism of the stable quotient.  
As an application, we prove that the property of having fully irreducible monodromy for a splitting of a hyperbolic free-by-cyclic group depends only on the component of the BNS-invariant containing the associated homomorphism to the integers.
\end{abstract}

\thanks{The first author was partially supported by the NSF
  postdoctoral fellowship, NSF MSPRF no. 1204814. The second author
  was partially supported by the NSF grant DMS-0904200 and by
  the Simons Foundation Collaboration grant no. 279836. The third
  author was partially supported by the NSF grant DMS-1207183 and acknowledges support from NSF grants DMS 1107452, 1107263, 1107367 ``RNMS: Geometric structures And Representation varieties" (the GEAR Network).}


\maketitle


\section{Introduction} \label{S:intro}

In the theory of $\Out(\FN)$ train-tracks serve as important tools for understanding free group automorphisms: given an automorphism $\phi$ one strives to find a train track representative (say, via the Bestvina--Handel algorithm) that is useful in analyzing the automorphism. 

In \cite{DKL2}, we naturally encountered train-track maps $f \colon \Theta \to \Theta$ for which $f_* \colon \pi_1(\Theta) \to \pi_1(\Theta)$ was {\sc not} injective (and thus also {\sc not} surjective by the Hopfian property of free groups); other sources that have considered train tracks for endomorphisms of free groups include \cite{DV96,Reynolds,AR}. We showed in \cite{DKL2} that $f_*$ descends to an injective endomorphism $\phi \colon Q \to Q$ of the \emph{stable quotient}
\[ Q = \pi_1(\Theta)/\bigcup_{k \geq 1} \ker(\phi^k).\]
The group $Q$ is also a nontrivial (since $f$ is a train-track map) free group, and in the setting of \cite{DKL2} $\phi$ is often an automorphism.  In this paper, we explain how to produce from any expanding, irreducible train track map $f \colon \Theta \to \Theta$ an honest train track representative $\bar f \colon \bar \Theta \to \bar \Theta$ for $\phi$, and we describe its relationship with $f$.

\begin{theorem} \label{T:promoting train track maps}
Let $f\colon\Theta\to\Theta$ be an expanding irreducible train track map. Let $f_*\colon \pi_1(\Theta)\to \pi_1(\Theta)$ be the free group endomorphism represented by $f$, and let $\phi\colon Q \to Q$ be the induced injective endomorphism of the stable quotient $Q$ of $f_*$.

Then there exists a finite graph $\bar \Theta$ with $\pi_1(\bar \Theta)\cong Q$ (and no valence $1$ vertices), and an expanding irreducible train-track map $\bar f\colon \bar \Theta \to \bar \Theta$ such that  $\bar f_* = \phi$, up to post-composition with an inner automorphism of $Q$.
Furthermore, there exists graph maps $\bar p \colon \bar \Theta \to \Theta$ and $\Phi \colon \Theta \to \bar \Theta$ such that 
\begin{itemize}
\item $\bar f \Phi = \Phi f$ and $\bar p \bar{f} = f \bar p$, and
\item $\bar p \Phi = f^K$ and $\Phi \bar p = \bar f^K$, for some $K \geq 1$.
\end{itemize}
\end{theorem}

As an application, we have the following theorem about the \emph{Bieri-Neumann-Strebel invariant} for free-by-cyclic groups (see \cite{BNS,Levitt,BG04,CL} for background information on the BNS-invariant).  
To state it, recall that a group homomorphism $u\in \Hom(G,\R) = H^1(G;\R)$ is \emph{primitive integral} if $u(G) = \Z$ and that the \emph{monodromy} $\phi_u\in \Out(\ker(u))$ of such a homomorphism is the generator of the action of $\Z$ on $\ker(u)$ defining the semi-direct product structure $G = \ker(u) \rtimes_{\phi_u}\Z$. Recall also that the BNS-invariant $\Sigma(G)$ of $G$ \cite{BNS} is an open subset of the positive projectivization,
\[ \Sigma(G) \subset (H^1(G;\R) - \{0\})/\R_+, \]
which captures finite generation properties; for example, a primitive integral class $u\in H^1(G;\R)$ has $\ker(u)$ finitely generated if and only if $u,-u\in \Sigma(G)$.

\begin{theorem} \label{T:same component, all iwip}
Suppose $G$ is a hyperbolic group, $\Sigma_0(G)$ a component of the BNS-invariant, and $u_0,u_1 \in H^1(G;\R)$ primitive integral classes projecting into $\Sigma_0(G)$ with $\ker(u_0),\ker(u_1)$ finitely generated.  Then $\ker(u_0)$ is free with fully-irreducible monodromy $\phi_{u_0}$ if and only if $\ker(u_1)$ is free with fully irreducible monodromy $\phi_{u_1}$. 
\end{theorem}

The fact that $\ker(u_0)$ is free if and only if $\ker(u_1)$ is free follows from \cite{GMSW}.  The point of the theorem is that the monodromy of $u_0$ is fully irreducible if and only if the monodromy for $u_1$ is.  
The proof of Theorem~\ref{T:same component, all iwip} builds on our papers  \cite{DKL,DKL2} which developed new machinery for studying dynamical aspects of free-by-cyclic groups by exploiting properties of natural semi-flows on associated folded mapping tori 2--complexes; see also \cite{AHR} for related work.

Since full irreducibility is preserved by taking inverses, Theorem~\ref{T:same component, all iwip} yields  the following corollary.

\begin{corollary} \label{C:connected union, all iwip}
Suppose $G$ is a hyperbolic group and that $\Sigma(G) \cup -\Sigma(G)$ is connected.  Then for any two primitive integral $u_0,u_1 \in H^1(G;\R)$ with finitely generated, free kernels, $\phi_{u_0}$ is fully irreducible if and only if $\phi_{u_1}$ is fully irreducible.
\end{corollary}
\begin{proof}
Consider a component $C$ of $\Sigma(G)$. By Theorem~\ref{T:same component, all iwip}, either every primitive integral $u\in H^1(G;\R)$ projecting into $C$ with $\ker(u)$ finitely generated has the property that $\ker(u)$ is free and $\phi_u$ is fully irreducible, or else no such $u$ projecting into $C$ has this property. Say that $C$ is a \emph{fully irreducible component} in the former case and that it is a \emph{non-fully irreducible component} in the latter.
Now if $\Sigma_0(G)$ is a fully irreducible component and $\Sigma_1(G)$ a non-fully irreducible component, then observe that $(\Sigma_0(G) \cup -\Sigma_0(G)) \cap (\Sigma_1(G) \cup -\Sigma_1(G)) = \emptyset$.  For, if not, then there exists a primitive integral $u$ with finitely generated kernel and $\phi_u$ fully irreducible, such that $-u$ lies in $\Sigma_1(G)$.  Since $\phi_u$ is fully irreducible if and only $\phi_{-u} = \phi_u^{-1}$ is, this is a contradiction.

Now let $\mathcal F(G) \subset \Sigma(G) \cup -\Sigma(G)$ denote the union of open sets $\Sigma_0(G) \cup -\Sigma_0(G)$, over all fully irreducible components $\Sigma_0(G)$, and let $\mathcal N(G) \subset \Sigma(G) \cup -\Sigma(G)$ be the union of open sets $\Sigma_1(G) \cup -\Sigma_1(G)$ over all non-fully irreducible components $\Sigma_1(G)$.  The open sets $\mathcal F(G)$ and $\mathcal N(G)$ cover $\Sigma(G) \cup -\Sigma(G)$ and are disjoint by the previous paragraph, hence one must be empty and the corollary follows.
\end{proof}

For the case that $G = \pi_1(M)$, where $M$ is a finite volume hyperbolic $3$--manifold, considerations of the Thurston norm \cite{ThuN} imply that $\Sigma(G)=-\Sigma(G)$ is projectively equal to a finite union of top-dimensional faces of the polyhedral Thurston norm ball in $H^1(M;\R)$ (c.f.~\cite{BNS}); thus here $\Sigma(G)\cup -\Sigma(G)$ is never connected unless it is empty.  However, for hyperbolic free-by-cyclic groups $G$ it can easily happen that $\Sigma(G) \cup -\Sigma(G)$ is connected and nonempty: In the main example of \cite{DKL2}, one may easily apply Brown's algorithm \cite[Theorem 4.4]{Brown} to the presentation \cite[Equation 3.4]{DKL2} to calculate that $\Sigma(G)$ contains all rays in $H^1(G;\R)\cong\R^2$ except for those in the directions $(-1,0)$, $(1,2)$, and $(1,-2)$ (as in \cite{DKL2}, we work with left actions, so we must take the negative of the result of applying Brown's algorithm). 
The cone $\EuScript{S}$ calculated in \cite[Example 8.3]{DKL2} is one component of $\Sigma(G)$, and the vector $u_1=(-1,2)\in \Sigma(G)$ satisfies $-u_1\notin\Sigma(G)$; see \cite[Figure 8]{DKL2}. In particular, we see that $\Sigma(G)\cup-\Sigma(G)$ is the entire positive projectivization of $H^1(G;\R)\setminus \{0\}\cong \R^2\setminus \{0\}$, and is thus connected.

Theorem~\ref{T:same component, all iwip} extends and generalizes our earlier result \cite[Theorem~C]{DKL}. There we considered a hyperbolic free-by-cyclic group $G=F_N\rtimes_{\phi_0} \Z$ with fully irreducible monodromy $\phi_0\in \Out(F_N)$ and constructed an open convex cone $\mathcal A\subseteq H^1(G;\R)$ containing the projection $F_N\rtimes_{\phi_0}\Z\to\Z$ and whose projectivization is contained in $\Sigma(G) \cap -\Sigma(G)$. 
Among other things, \cite[Theorem C]{DKL} showed that for every primitive integral $u\in \mathcal A$ the splitting $G = \ker(u)\rtimes_{\phi_u}\Z$ has finitely generated free kernel $\ker(u)$ and fully irreducible monodromy $\phi_u\in\Out(\ker(u))$.

The proofs of \cite[Theorem~C]{DKL} and  Theorem~\ref{T:same component, all iwip} are fairly different, although both exploit the dynamics of a natural semi-flow on the \emph{folded mapping torus} $X_f$ constructed from a train-track representative $f\colon\Gamma\to\Gamma$ of $\phi_0$. Our proof of \cite[Theorem~C]{DKL} starts by establishing the existence of a cross-section $\Theta_u\subseteq X_f$ dual to each primitive integral $u\in\mathcal A$ such that the first return map $f_u\colon\Theta_u\to\Theta_u$ is a train-track representative of $\phi_u$. We then used the fine structure of the semi-flow (derived from the train map $f$ and the fully irreducible atoroidal assumption on $\phi_0$) to conclude that $f_u$ is expanding and irreducible and has connected Whitehead graphs for all vertices of $\Theta_u$. 
This, together with the word-hyperbolicity of $G$, allowed us to apply a criterion obtained in \cite{K14} to conclude that $\phi_u$ is fully irreducible.

The proof of Theorem~\ref{T:same component, all iwip} starts similarly. Given $G=F_N\rtimes_{\phi_0} \Z$ as above and an epimorphism $u\colon G\to\Z$ in the same component of $\Sigma(G)$ as $F_N\rtimes_{\phi_0}\Z\to\Z$ and with $\ker(u)$ being finitely generated (and hence free), we use our results from \cite{DKL2} to find a section $\Theta_u\subseteq X_f$ dual to $u$ such that the first return map $f_u\colon\Theta_u\to\Theta_u$ is an expanding irreducible train track map. However, now $(f_u)_\ast$ is a possibly non-injective endomorphism of $\pi_1(\Theta_u)$. We thus pass to the stable quotient of $(f_u)_*$, which we note is equal to the monodromy automorphism $\phi_u\in\Out(\ker(u))$ since $\ker(u)$ is finitely generated. We then apply Theorem~\ref{T:promoting train track maps} to obtain an expanding irreducible train-track representative $\bar{f}_u\colon\bar\Theta_u\to \bar\Theta_u$ and use the provided maps $\bar\Theta_u \leftrightarrows \Theta_u$ to construct a pair of flow-equivariant homotopy equivalences $M_{\bar{f}_u} \leftrightarrows X_f$ with additional nice properties; here $M_{\bar{f}_u}$ is the mapping torus of $\bar{f}_u$. Supposing that $\phi_u = (\bar{f}_u)_*$ were not fully irreducible, we then find a proper nontrivial flow-invariant subcomplex in a finite cover of $M_{\bar{f}_u}$ which, via the equivalences $M_{\bar{f}_u}\leftrightarrows X_f$, gives rise to a proper nontrivial flow-invariant subcomplex of some finite cover of $X_f$. From here we deduce the existence of a finite cover $\Delta\to\Gamma$ and a lift $h\colon\Delta\to\Delta$ of some positive power of $f$ such that $\Delta$ admits a proper nontrivial $h$--invariant subgraph. But by a general result of Bestvina--Feighn--Handel \cite{BFH97}, this conclusion contradicts the assumption that $\phi_0 = f_*$ is fully irreducible.

Our proof of Theorem~\ref{T:same component, all iwip} uses the assumption that $u_1$ and $u_2$ lie in the same component of $\Sigma(G)$ to conclude, via the results of \cite{DKL2}, that both splittings of $G$ come from cross sections of a single $2$--complex equipped with a semi-flow. It is therefore unlikely that this approach will lead to any insights regarding splittings in different components of $\Sigma(G)$. Nevertheless, we ask:

\begin{question}
Can Theorem~\ref{T:same component, all iwip} be extended to remove the hypothesis that $u_1$ and $u_2$ lie in the same component of the BNS-invariant $\Sigma(G)$?
\end{question}

\noindent {\bf Acknowledgements:} The authors would like to thank the referee for carefully reading an earlier version of the paper and providing helpful suggestions that improved the exposition.

\section{Induced train track maps -- general setting} \label{S:induced_tt}

Let $\Theta$ be a finite graph with no valence $1$ vertices, and let $f\colon\Theta\to\Theta$ be a graph map (as in \cite[Definition 2.1]{DKL}). Recall from \cite[\S2]{DKL} that the $(e',e)$--entry of the \emph{transition matrix} $A(f)$ of $f$ records the total number of occurrences of the edge $e^{\pm1}$ in the edge path $f(e')$. The transition matrix $A(f)$ is \emph{positive} (denoted $A(f) > 0$) if every entry is positive and is \emph{irreducible} if for every ordered pair $(e',e)$ of edges of $\Theta$ there exists $t\ge 1$ such that the $(e',e)$--entry of $A(f)^t$ is positive. We say that $f$ is \emph{irreducible} if its transition matrix $A(f)$ is irreducible, and that $f$ is \emph{expanding} if for each edge $e$ of $\Theta$ the edge paths $f^n(e)$ have combinatorial length tending to $\infty$ with $n$. In this paper, as in \cite{DKL2}, we use the term ``train-track map'' to mean the following:

\begin{defn}[Train-track map]\label{def:train-track}
A \emph{train-track map} is a graph map $f\colon \Theta\to \Theta$ such that:
\begin{itemize}
\item the map $f$ is surjective, and
\item for every edge $e$ of $\Theta$ and every $n\ge 1$ the map $f^n\vert_e$ is an immersion.
\end{itemize}
\end{defn}

Note that, unlike the original definition \cite{BH92}, our definition of train-track maps allows for valence $2$ vertices in $\Theta$. Lemma 2.12 of \cite{DKL} shows that train-track maps must be locally injective at each valence $2$, thus the presence of valence $2$ vertices does not lead to any complications.

Our Definition~\ref{def:train-track} differs from the traditional setting in another important way; namely, we do not require a train-track map $f\colon\Theta\to\Theta$ to be a homotopy equivalence. 
Thus $f_*$ need only determine an endomorphism of $\pi_1(\Theta)$, in which case $f$ is not a topological representative of any outer automorphism of $\pi_1(\Theta)$.

Nevertheless in \cite[\S4]{DKL2} we saw that an arbitrary endomorphism $\varphi\colon F_N\to F_N$ of a finite-rank free group naturally gives rise to an \emph{injective} endomorphism $\bar{\varphi}$ of the quotient group
\[ Q = F_N/\bigcup_{k \geq 1} \ker(\varphi^k). \]
In fact, the kernels stabilize after finitely many, say $K$, steps so that $\bigcup_{k \geq 1} \ker(\varphi^k) = \ker(\varphi^K)$.  Then $Q$ is isomorphic to the image $J = \varphi^K(F_N) < F_N$  and is thus itself free.  Moreover, the isomorphism conjugates $\bar \varphi$ to the restriction of $\varphi$ to $J$, and thus we may view $\bar \varphi \colon Q \to Q$ and $\varphi|_J \colon J \to J$ as the ``same'' injective endomorphism.

We refer to the train track map $f \colon \Theta \to \Theta$ as a {\em weak train track representative} of this quotient endomorphism $\bar \varphi \colon Q \to Q$ of $f_*$.   The goal of this section is to prove Theorem~\ref{T:promoting train track maps} which promotes the weak train track representative $f \colon \Theta \to \Theta$ to an honest train track representative $\bar f \colon \bar \Theta \to \bar \Theta$ of $\bar\varphi$ (meaning that $\bar f_* = \bar \varphi$ up to conjugation) whenever $f$ is an expanding irreducible train track map.

\subsection{Subgroups and lifts}
\label{sec:subgroups_and_lifts}

For the remainder of \S\ref{S:induced_tt} we fix an expanding irreducible train track map $f\colon\Theta\to\Theta$. We begin with a simple observation.

\begin{lemma}
\label{lem:legal_loops}
For every edge $e$ of $\Theta$, there exists a legal loop $\alpha_e\colon S^1\to \Theta$ crossing $e$. Here ``legal'' simply means that $f^k\circ \alpha\colon S^1\to \Theta$ is an immersion for all $k\ge 0$. In particular, $\Theta$ is a union of legal loops.
\end{lemma} 
\begin{proof}
Since $f$ is expanding and $\Theta$ has finitely many edges, there exists an integer $j$ so that $f^j(e)$ crosses some edge $e'$ at least twice in the same direction. Irreducibility then provides some $\ell\ge j$ so that $f^\ell(e)$ crosses $e$ twice in the same direction. Thus we may find a subinterval $I\subset e$, say whose endpoints both map to an interior point of $e$, such that the restriction $f^\ell\vert_I$ defines an immersed closed loop $\alpha \colon S^1\to \Theta$ crossing $e$. Since $f$ is a train-track map, it follows that $\alpha$ is legal.
\end{proof}

Let $v$ be an $f$--periodic vertex of $\Theta$, say of period $r$. Then set $v_0 = v$ and $v_i = f^i(v_0)$ for $i = 1,\ldots, r-1$.  We consider the indices of the vertices $v_0,\ldots,v_{r-1}$ modulo $r$ in what follows.

Now we let $B_i = \pi_1(\Theta,v_i)$. Then $f$ induces homomorphisms $B_i \to B_{i+1}$, with $i = 0,\ldots,r-1$ and indices modulo $r$.  We write $f_*$ to denote any of these homomorphisms (though to clarify, we may also write $(f_*)_i \colon B_i \to B_{i+1}$).  With this convention, we can write $f_*^j$, for $j \in \Z$ with $j \geq 0$, to denote any of the $r$ homomorphisms $(f_*^j)_i\colon B_i \to B_{i+j}$ with subscripts taken modulo $r$.

A path $\delta$ from $v_j$ to $v_i$ determines an isomorphism $\rho_\delta \colon B_i \to B_j$.  The image $f_u^\ell(\delta) = \delta'$ likewise determines an isomorphism $\rho_{\delta'} \colon B_{i+\ell} \to B_{j+\ell}$, and we have
\begin{equation} \label{E:nearly_conjugate} 
(f_*^\ell)_j \circ \rho_{\delta} = \rho_{\delta'} \circ (f_*^\ell)_i.
\end{equation}
Note that changing $\delta$ (and hence also $\delta'$), we obtain potentially different isomorphisms $\rho_\delta$ and $\rho_{\delta'}$.  

Fix $i$ and let $n > 0$ be an integer such that the restriction of $f_*$ to the subgroup $J_i = f_*^{nr}(B_i) <  B_i$ is injective.  Let $\delta$ be a path from $v_{i+1}$ to $v_i$ and $\delta' = f^{nr}(\delta)$.  Then setting
\[ J_{i+1}= f_*^{nr}(B_{i+1}) \]
we have
\[  J_{i+1} = f_*^{nr}(\rho_\delta(B_i)) = \rho_{\delta'}(f_*^{nr}(B_i)) = \rho_{\delta'}(J_i),\]
and hence $\rho_{\delta'}$ restricts to an isomorphism from $J_i$ to $J_{i+1}$.  It is interesting to note that $J_{i+1}$ is defined without reference to $\delta$ (or $\delta'$).
Furthermore, if $\delta'' = f(\delta')$, then by \eqref{E:nearly_conjugate} we have
\[ (f_*)_{i+1} = \rho_{\delta''} \circ (f_*)_i  \circ \rho_{\delta'}^{-1} \colon B_{i+1} \to B_{i+2}, \]
and hence the restriction of $f_*$ to $J_{i+1}$ is injective.
Therefore, if we let $n(i) > 0$ be the smallest positive integer so that $f_*$ restricted to $J_i = f_*^{n(i)r}(B_i)$ is injective, then we have shown that $n(i) \geq n(i+1)$.  Since this condition is true for all $i$, it follows that $n(i) = n(j)$ for all $0 \leq i,j \leq r-1$. We henceforth fix  $n = n(i)$.  

For each $i$ let $p_i \colon \widetilde \Theta_i \to \Theta$ denote the cover corresponding to the conjugacy class $J_i < \pi_1(\Theta,v_i)$.   Let $\widetilde V_i \subset p_i^{-1}(v_i)$ denote the set of all vertices $\widetilde v_i$ so that $(p_i)_*(\pi_1(\widetilde \Theta_i,\widetilde v_i)) = J_i$.  Then the covering group of $p_i \colon \widetilde \Theta_i \to \Theta$ acts simply transitively on $\widetilde V_i$.   Since the isomorphism $\rho_{\delta'}$ sends $J_i$ to $J_{i+1}$, it follows that there is an isomorphism of covering spaces $\widetilde \Theta_i \to \widetilde \Theta_{i+1}$.  Repeating this $r$ times, we see that all the covering spaces $\{ p_i \colon \widetilde \Theta_i \to \Theta \}_{i=0}^{r-1}$ are pairwise isomorphic.  In particular, we now simply write $p \colon \widetilde \Theta \to \Theta$ for any one of these spaces. Write $\bar \Theta$ for the convex (Stallings) core of $\widetilde \Theta$, and we note that this is a proper subgraph.

For all $m\ge n$ we have $f_*^{mr}(B_i) = f_*^{nr}(f_*^{(m-n)r}(B_i)) \le f_*^{nr}(B_i) = J_i$. Thus from standard covering space theory, we know that for every $i$ and every $\widetilde v_i \in \widetilde V_i$ there is a unique continuous map $\widehat{f_{\widetilde v_i}^{mr}}$ making the following diagram commute:
\[ \xymatrix{  & (\widetilde \Theta, \widetilde v_i) \ar[d]^p \\
(\Theta,v_i) \ar[r]^{f^{mr}} \ar[ur]^{\widehat{f_{\widetilde v_i}^{mr}}} & (\Theta,v_i) }\]

\begin{proposition} \label{P:surjective lifted power}
For any $m \geq n$ and $\widetilde v_i\in \widetilde V_i$, we have $\widehat{f_{\widetilde v_i}^{mr}}(\Theta) =\bar \Theta$. 
\end{proposition}
\begin{proof}
Fix $m\ge n$ and $\widetilde v_i\in\widetilde V_i$. Since $\widehat{f_{\widetilde v_i}^{nr}}$ is surjective on the level of fundamental groups, the containment $\bar \Theta \subseteq \widehat{f_{\widetilde v_i}^{nr}}(\Theta)$ is immediate. Since $f\colon \Theta\to\Theta$ is itself surjective, it follows that we also have the inclusion
\[\bar\Theta \subseteq \widehat{f_{\widetilde v_i}^{nr}}(\Theta) = \widehat{f_{\widetilde v_i}^{nr}}\circ f^{(m-n)r}(\Theta) = \widehat{f_{\widetilde v_i}^{mr}}(\Theta).\]
Here we have used the equality $\widehat{f_{\widetilde v_i}^{nr}} \circ f^{(m-n)r} = \widehat{f_{\widetilde v_i}^{mr}}$ guaranteed by the uniqueness of lifts of $f^{mr}$ sending $v_i$ to $\widetilde v_i$.

On the other hand, for any legal loop $\alpha\colon S^1\to \Theta$ the composition $\widehat{f_{\widetilde v_i}^{mr}} \circ \alpha$ is an immersion; this conclusion follows from the local injectivity of $p\circ \widehat{f^{mr}_{\widetilde v_i}}\circ \alpha = f^{mr} \circ \alpha$. 
Since the closure of $\widetilde \Theta \setminus \bar\Theta$ consists of finitely may pairwise disjoint trees,  it follows that the image of $\widehat{f_{\widetilde v_i}^{mr}}\circ \alpha$ must be contained in the core $\bar\Theta$. The containment $\widehat{f_{\widetilde v_i}^{mr}}(\Theta)\subset \bar\Theta$ now follows from the fact that $\Theta$ is a union of legal loops (Lemma~\ref{lem:legal_loops}).
\end{proof}

Since $(f_*)_i$ restricted to $J_i$ is an injective homomorphism into $J_{i+1}$, for any choice of basepoints $\widetilde v_i \in \widetilde V_i$ and $\widetilde v_{i+1} \in \widetilde V_{i+1}$ covering space theory again provides a unique map $\widetilde f_{\widetilde v_i, \widetilde v_{i+1}} \colon \widetilde \Theta \to \widetilde \Theta$ making the following diagram commute:
\[ \xymatrix{ (\widetilde \Theta,\widetilde v_i) \ar[r]^{\widetilde f_{\widetilde v_i, \widetilde v_{i+1}} \, \, } \ar[d]_p & (\widetilde \Theta,\widetilde v_{i+1}) \ar[d]^p\\
(\Theta,v_i) \ar[r]^{f \, \, } & (\Theta,v_{i+1})} \]

\begin{proposition}\label{prop:lifted_tt_map}
Let $\bar f = \widetilde f_{\widetilde v_i, \widetilde v_{i+1}}\vert_{\bar \Theta}$ be the restriction of any such lift $\widetilde f_{\widetilde v_i, \widetilde v_{i+1}}$ to $\bar \Theta$. Then $\bar f(\bar \Theta) = \bar\Theta$ and $\bar f \colon \bar \Theta\to\bar\Theta$ is an expanding train track map.  
\end{proposition}
\begin{proof}
Proposition~\ref{P:surjective lifted power} and Lemma~\ref{lem:legal_loops} show that there exist finitely many legal loops $\alpha_1,\dotsc,\alpha_k\colon S^1\to \Theta$ such that $\bar\Theta$ is the union of the images of $\beta_j = \widehat{f_{\widetilde v_i}^{nr}} \circ \alpha_j$ for $j = 1,\dotsc, k$. Noting that $\bar f \circ \beta_j$ is an immersion (because it is a lift of the immersion  $f\circ f^{nr} \circ \alpha_j$), its image must be contained in $\bar \Theta$.  Therefore, $\bar f$ maps the union $\cup_j \beta_j(S^1) =\bar\Theta$ into $\bar\Theta$, and we conclude $\bar f(\bar \Theta) \subseteq \bar \Theta$.

Thus $\bar f$ is a graph map from $\bar\Theta$ to itself, $\bar f \colon \bar \Theta \to \bar \Theta$, and we may consider its iterates $\bar f^\ell$. As above, we now see that $\bar f^\ell \circ \beta_j$ lifts $f^\ell\circ f^{nr}\circ \alpha_j$ and so is an immersion for each $\ell > 0$. Since each edge of $\bar\Theta$ is crossed by some $\beta_j$, this proves each iterate $\bar f^\ell$ is locally injective on each edge $\tilde{e}$ of $\bar\Theta$. Moreover, since $p$ is a covering map, the combinatorial length of $\bar f^\ell(\tilde{e})$ is equal to that of $p\circ \bar f^\ell(\tilde{e}) = f^\ell(p(\tilde{e}))$. Therefore $\bar f$ is expanding because $f$ is.

To prove the proposition it remains to show that $\bar f(\bar\Theta)\supseteq\bar\Theta$. Fix preferred lifts $\widetilde v_i\in \widetilde V_i$ for each $0\le i < r$ and set $\bar f_i = \widetilde f_{\widetilde v_i, \widetilde v_{i+1}}\vert_{\bar \Theta}$ for $0 \le i < r$. It suffices to show that each $\bar f_i$ maps $\bar\Theta$ onto $\bar \Theta$. To see that $\bar f_i(\bar\Theta)=\bar \Theta$, note that
\[\bar f_i \circ \bar f_{i-1} \circ \dotsb \circ \bar f_{i+2}\circ \bar f_{i+1}\circ \widehat{f^{nr}_{\widetilde v_{i+1}}}\colon(\Theta,v_{i+1})\to (\widetilde \Theta, \widetilde v_{i+1})\]
(with subscripts taken modulo r) is a lift of $f^{(n+1)r}$ taking $v_{i+1}$ to $\widetilde v_{i+1}$. Therefore the above composition (and in particular $\bar f_i$) has image $\bar \Theta$ by Proposition~\ref{P:surjective lifted power}.
\end{proof}

For the remainder of this section, we let $\widetilde f = \widetilde f_{\widetilde v_i, \widetilde v_{i+1}}$ be any lift of $f$ as above, let $\bar f = \widetilde f \vert_{\bar\Theta} \colon \bar \Theta \to \bar \Theta$ be its restriction to the core $\bar\Theta$ of the covering $p \colon \widetilde \Theta \to \Theta$, and write $\bar p = p|_{\bar \Theta} \colon \bar \Theta \to \Theta$.

\begin{lemma} \label{L:semi-conjugating lift} There is a lift $\Phi \colon \Theta \to \widetilde \Theta$ of a power $f^K$ of $f$ with $\Phi(\Theta) = \bar \Theta$ such that $\Phi \circ p =  \widetilde f^K$ and consequently $\Phi \circ \bar p = \bar f^K$.  
\end{lemma}
Because $\Phi$ is a lift of $f^K$ and since $\bar f$ and $\bar p$ are restrictions, we also obviously have $p \circ \Phi = f^K$, $\bar p \circ \Phi = f^K$, and $\bar p \circ \bar f = f \circ \bar p$. 
\begin{proof}
The composition $\bar f^r$ necessarily maps the finite set $p^{-1}(v_i)\cap\bar\Theta$ into itself. Thus the sequence $\widetilde v_i, \bar f^r(\widetilde v_i), \bar f^{2r}(\widetilde v_i),\dotsc$ is eventually periodic. Choosing $k$ to be a sufficiently large multiple of the period, it follows that the point $z\colonequals \bar f^{kr}(\widetilde v_i)$ satisfies $\bar f^{mkr}(z) = z$ and $\bar f^{mkr}(\widetilde v_i) = z$ for all $m\ge1$. 

Set $J = p_*(\pi_1(\widetilde \Theta, z))$, and note that $J$ and $J_i=p_*(\pi_1(\widetilde \Theta,\widetilde v_i))$ are conjugate but possibly distinct subgroups of $B_i$. Observe that
\begin{align*}
f^{2knr}_*(\pi_1(\Theta,v_i)) &= f_*^{knr}\circ f_*^{knr}(\pi_1(\Theta,v_i)) \le f_*^{knr}(J_i) = f_*^{knr}\circ p_*(\pi_1(\widetilde\Theta,\widetilde v_i))\\ &= p_*\circ\widetilde f_*^{knr}(\pi_1(\widetilde\Theta,\widetilde v_i)) \le p_*(\pi_1(\widetilde\Theta,z)) = J.
\end{align*}
Therefore there is a unique lift $\Phi\colon(\Theta, v_i)\to (\widetilde\Theta,z)$ of $f^{2knr}$ sending $v_i$ to $z$. By inspection, this lift must be $\Phi = \widetilde f^{knr}\circ \widehat{f_{\widetilde v_i}^{knr}} = \bar f^{knr}\circ \widehat{f_{\widetilde v_i}^{nr}}$ and therefore has image $\bar\Theta$ by Propositions~\ref{P:surjective lifted power}--\ref{prop:lifted_tt_map}. 

Set $K = 2knr$, and we claim that $\widetilde{f}^K = \Phi \circ p$. Indeed, both maps lift the composition
\[f^K \circ p\colon (\widetilde \Theta,\widetilde v_i) \to (\Theta,v_i)\]
and send $\widetilde v_i\to z$ by construction; hence they are equal by uniqueness of lifts. Interestingly, this argument shows that a power of $\widetilde{f}$ (namely $\widetilde{f}^K$) maps all of $\widetilde{\Theta}$ into $\bar\Theta$.  
\end{proof}

\begin{proposition}  \label{P:lift_is_irred}
Let $f \colon \Theta \to \Theta$ and $\bar f \colon \bar \Theta \to \bar \Theta$ be as above.
If $f$ is irreducible then $\bar f$ is irreducible. If $f$ has a power with positive transition matrix, then $\bar f$ has a power with positive transition matrix.
\end{proposition}
\begin{proof}
Assume first that $f$ is irreducible. Choose arbitrary edges $\tilde{e},\tilde{e}'$ of $\bar\Theta$ and set $e = \bar p(\tilde{e})$. With $\Phi$ as in Lemma~\ref{L:semi-conjugating lift}, we have $\Phi(\Theta) = \bar\Theta$, and so we may choose an edge $e_0$ of $\Theta$ such that $\Phi(e_0)\supseteq \tilde{e}'$. By irreducibility of $f$, there exist $s> 0$ such that $e_0\subseteq f^s(e)$.  Then applying Lemma~\ref{L:semi-conjugating lift} with $K$ as in the statement, we have
\begin{align*}
\bar f^{K+s}(\tilde e) &= \bar f^K\circ \bar f^s(\tilde e) 
= \Phi \circ \bar p \circ  \bar f^s(\tilde e)
= \Phi \circ f^s(e)
\supseteq \Phi(e_0)
\supseteq \tilde{e}'.
\end{align*}
Thus $\bar f$ is irreducible provided $f$ is. Next assume there is a power $f^\ell$ with positive transition matrix, so that in particular $f^\ell(e) = \Theta$ for every edge $e$ of $\Theta$. Choosing any edge $\tilde e$ of $\bar\Theta$, as above we find
\begin{align*}
\bar f^{K+\ell}(\tilde e) &= \bar f^K\circ \bar f^\ell(\tilde e) 
= \Phi\circ \bar p \circ  \bar f^\ell(\tilde e)
= \Phi\circ f^\ell(\bar p(\tilde e))
= \Phi(\Theta)
= \bar\Theta.
\end{align*}
Therefore $\bar f^{K+\ell}$ has positive transition matrix as well.
\end{proof}

\subsection{Train tracks for induced endomorphisms} 

Combining the results above, we can now easily give the
\begin{proof}[Proof of Theorem~\ref{T:promoting train track maps}.]
The map $\bar f \colon \bar \Theta \to \bar \Theta$ is given by Proposition \ref{prop:lifted_tt_map}, which together with Proposition \ref{P:lift_is_irred} implies $\bar f$ is an expanding irreducible train track map.  The map $\bar p \colon \bar \Theta \to \Theta$ is the restriction of a covering map to the core, and hence $\bar p_*$ defines an isomorphism of $\pi_1(\bar \Theta)$ onto the image $J = f^{nr}_*(\pi_1(\Theta)) < \pi_1(\Theta)$, up to conjugation.
By construction, $f_*|_J$ determines an injective endomorphism $J \to J$, up to conjugation.  Since $\bar p_* \bar f_* = f_* \bar p_*$, it follows that $\bar f_*$ induces an injective endomorphism of $\pi_1(\bar \Theta)$, up to conjugation.  As was shown in \cite[Proposition 2.6]{DKL2}, there is an isomorphism $J \to Q$ conjugating $f_*|_J$ to $\phi$.  It follows that with respect to this isomorphism and $\bar p_*$ we have $\phi = \bar f_*$, up to conjugation.

Let $\Phi \colon \Theta \to \bar \Theta$ and $K > 0$ be as in Lemma~\ref{L:semi-conjugating lift}.  The conclusion of that lemma proves the remainder of the theorem.
\end{proof}

The intrepid reader is encouraged to apply Theorem~\ref{T:promoting train track maps} to the naturally arising first return map $f_2\colon \Theta_2\to \Theta_2$ described in Example 5.7 and Figure 7 of \cite{DKL2}. For a warm-up, here is a simpler example:

\begin{example}
Let $\Theta$ be the $3$--petal rose depicted in Figure~\ref{F:example_train_track}, and let $f\colon \Theta\to\Theta$ be the expanding irreducible train track map defined on edges by $f(a) = ab$, $f(b) = bc$, $f(c) = abbc$. Then $\pi_1(\Theta)$ is free on generators ${\bf a, b, c}$ (correspoding to petals of the same letter), and we find that $f_*(\pi_1(\Theta)) = \langle{\bf ab}, {\bf bc}\rangle$. Thus $f_*$ is neither surjective nor injective, but we find that the restriction of $f_*$ to $J = \langle{\bf ab}, {\bf bc}\rangle$ is injective. 
The induced endomorphism $\phi\colon Q\to Q$ of the stable quotient $Q\cong J$ of $f_*$ is therefore given by $\phi({\bf ab}) = ({\bf ab})({\bf bc})$ and $\phi({\bf bc}) = ({\bf bc})({\bf ab})({\bf bc})$, which is an \emph{automorphism} of this rank $2$ free group.

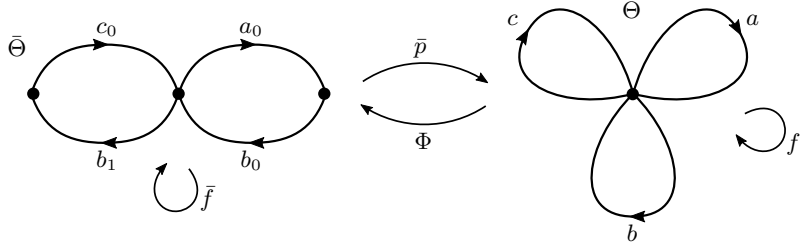
\begin{figure} [htb]
\labellist
\small\hair 2pt
\pinlabel $a$ [bl] at 233 63
\pinlabel $b$ [t] at 197 1
\pinlabel $c$ [br] at 161 63
\pinlabel $a_0$ [b] at 73 59
\pinlabel $b_0$ [t] at 73 25
\pinlabel $b_1$ [t] at 26 25
\pinlabel $c_0$ [b] at 26 59
\pinlabel $f$ [tl] at 246 30
\pinlabel $\bar{f}$ [tl] at 55 13
\pinlabel $\bar{p}$ [b] at 128 52
\pinlabel $\Phi$ [t] at 129 31
\pinlabel $\Theta$ [b] at 197 65
\pinlabel $\bar\Theta$ [br] at 1 53
\endlabellist
\begin{center}
\includegraphics[height=2.9cm]{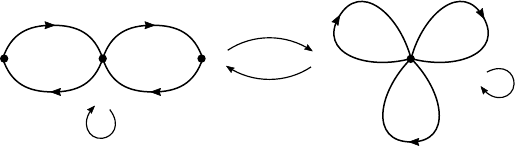} \caption{An application of Theorem~\ref{T:promoting train track maps}.}
\label{F:example_train_track}
\end{center}
\end{figure}

Plugging $f\colon \Theta\to \Theta$ into Theorem~\ref{T:promoting train track maps}, the construction produces the graph $\bar\Theta$ depicted in Figure~\ref{F:example_train_track} along with maps $\bar f$, $\bar p$, and $\Phi$ defined on edges by:
\begin{itemize}
\item $\bar{f}(a_0) = a_0b_0$, $\bar{f}(b_0) = \bar{f}(b_1) = b_1c_0$, and $\bar{f}(c_0) = a_0b_0b_1c_0$
\item $\bar{p}(a_0) = a$, $\bar{p}(b_0) = \bar{p}(b_1) = b$, and $\bar{p}(c_0) = c$
\item $\Phi(a) = a_0b_0$, $\Phi(b) = b_1c_0$, and $\Phi(c) = a_0b_0b_1c_0$.
\end{itemize}
One may easily verify that these satisfy the conclusion of Theorem~\ref{T:promoting train track maps} with $K = 1$.
\end{example}

\section{Semi-flows on $2$--complexes and free-by-cyclic groups} \label{S:semiflows}

To see how Theorem~\ref{T:promoting train track maps} can be applied to Theorem~\ref{T:same component, all iwip}, we briefly recall some of the setup and results from \cite{DKL,DKL2}.  Starting with an expanding, irreducible train-track map $f \colon \Gamma \to \Gamma$ representing an {\em automorphism} of the free group $\pi_1(\Gamma)$, in \cite{DKL} we constructed a $2$--complex $X = X_f$, the {\em folded mapping torus}, which is a (homotopy equivalent) quotient of the mapping torus of $f$ and contains an embedded copy of $\Gamma$.  The suspension flow on the mapping torus descends to a semi-flow $\psi$ on $X$ having $\Gamma$ as a cross section and $f$ as first return map, in the following sense.

\begin{defn}[see {\cite[\S5.1]{DKL2}}]
\label{def:cross_section}
A \emph{cross section} of $(X,\flow)$ is a finite embedded graph $\Theta\subset \Omega$ that is \emph{transverse to $\flow$} (meaning there is a neighborhood $W$ of $\Theta$ and a map $\eta\colon W\to S^1$ such that $\Theta = \eta\inv(x_0)$ for some $x_0\in S^1$ and for each $\xi\in X$ the map $\{s\in \R_{\ge 0} \mid \flow_s(\xi)\in W\}\to S^1$ given by $s\mapsto \eta(\flow_s(\xi))$ is an orientation preserving local diffeomorphism) with the property that every flowline hits $\Theta$ infinitely often (meaning $\{s\in \R_{\ge 0} \mid \flow_s(\xi)\in \Theta\}$ is unbounded for all $\xi\in X$).
\end{defn}

Being homotopy equivalent to the mapping torus, we have $G:= \pi_1(X) = \pi_1(\Gamma) \rtimes_{f_*} \Z$. The projection onto $\Z$ defines a primitive integral element $u_0 \in \Hom(G;\R)= H^1(G;\R) = H^1(X;\R)$.   The class $u_0$ projects into a component $\Sigma_0(G)$ of the BNS-invariant $\Sigma(G)$ of $G$, and we let $\mathcal S \subset H^1(G;\R)$ denote the open cone which is the preimage of $\Sigma_0(G)$. In \cite{DKL2} we proved that every primitive integral $u \in \mathcal S$ is ``dual'' to a cross section $\Theta \subset X$ of $\psi$ enjoying a variety of properties; 
see also \cite{Gau1,Gau3,Wang} for other results related to the existence of dual cross-sections for complexes equipped with semi-flows.
 To describe the duality, we recall that the first return map $f_\Theta \colon \Theta \to \Theta$ of $\flow$ to $\Theta$ allows us to write $G$ as the fundamental group of the mapping torus of $f_\Theta$.  This expression for $G$ determines an associated homomorphism to $\Z$ which is precisely $u$.  The class $u$ is determined by $\Theta$, and we thus write $[\Theta] = u$.  Alternatively, $\Theta$ is dual to $u$ if the map witnessing the transversality of $\Theta$ to $\flow$ can taken as a map $\eta_\Theta \colon X \to S^1$ defined on all of $X$ for which $(\eta_\Theta)_* = u$.  Then, $\flow$ can be reparameterized to $\flow^\Theta_s$ so that the time-one map, $\flow^\Theta_1$ restricted to $\Theta$ is the first return map. 

The map $f_\Theta$ was shown to be an expanding irreducible train-track map in \cite{DKL2}, but it is not a homotopy equivalence in general.  The descent to the stable quotient $\phi_{[\Theta]} \colon Q_{[\Theta]} \to Q_{[\Theta]}$ of $(f_{\Theta})_{*}$ is an automorphism if and only if $\ker([\Theta])$ is finitely generated.  In this case we can identify $Q_{[\Theta]} = \ker([\Theta])$, so that the associated splitting of $G$ as a semi-direct product $G = \ker([\Theta]) \rtimes \Z$ has monodromy $\phi_{[\Theta]}$.  The associated expanding irreducible train track map $\bar f_\Theta \colon \bar \Theta \to \bar \Theta$ from Theorem~\ref{T:promoting train track maps} is thus a topological representative for $\phi_{[\Theta]}$.  Therefore, Theorem~\ref{T:same component, all iwip} reduces to proving the following.

\begin{theorem} \label{T:other section iwip tech}
Suppose $f \colon \Gamma \to \Gamma$ is an expanding irreducible train track representative of a hyperbolic fully irreducible automorphism.  Further assume that $\Theta \subset X = X_f$ is a section of the semi-flow $\psi$, as constructed in \cite{DKL2}, with first return map $f_\Theta \colon \Theta \to \Theta$ such that $\ker([\Theta])$ is finitely generated.  Then for the induced train track map $\bar f_{\Theta} \colon \bar \Theta \to \bar \Theta$ from Theorem~\ref{T:promoting train track maps}, $(\bar f_{\Theta})_*$ is a fully irreducible automorphism.
\end{theorem}

\begin{proof}[Proof of Theorem~\ref{T:same component, all iwip} from Theorem~\ref{T:other section iwip tech}]
Suppose that $\ker(u_0)$, say, is free and $\phi_{u_0}$ is fully irreducible.  Let $f \colon \Gamma \to \Gamma$ be an expanding irreducible train track representative of $\phi_{u_0}$, and let $X,\psi$ be the associated folded mapping torus and suspension semi-flow.  From \cite{DKL2}, there is a section $\Theta \subset X$ such that $[\Theta] = u_1$ whose first return map $f_\Theta \colon \Theta \to \Theta$ has the property that $(f_{\Theta})_*$ descends to the monodromy $\phi_{u_1}$ on the (free) stable quotient $Q_{u_1} = \pi_1(\bar \Theta)$.  By Theorem~\ref{T:other section iwip tech}, $(\bar f_{\Theta})_* = \phi_{u_1}$ is fully irreducible, as required.
\end{proof}

The proof of Theorem~\ref{T:other section iwip tech} requires some new constructions which are carried out in the next few sections.
We need to work in a slightly more general context of semi-flows on compact $2$--complexes, without fixed points.  Cross sections and duality are defined just as above.

\section{Flow-equivariant maps}

Here we describe a general procedure for producing maps between spaces equipped with semi-flows.  
The particular quality of map we will require is  provided by the following:

\begin{defn}
Given spaces $X,Y$ each equipped with semi-flows $\flow^X_s,\flow^Y_s$, then maps $\alpha \colon X \to Y$ and $\beta \colon Y \to X$ are called {\em flow-homotopy inverse maps} if (1) the maps are flow-equivariant, i.e.
\[ \flow^Y_s \alpha = \alpha \flow^X_s \mbox{ and } \flow^X_s \beta = \beta \flow^Y_s \]
for all $s \geq 0$, and (2) there exists $K > 0$ so that $\beta \alpha = \flow^X_K$ and $\alpha \beta = \flow^Y_K$.
Note that $\alpha$ and $\beta$ are indeed homotopy inverses of each other (with the semi-flows defining the required homotopies).  We also call $\alpha$ and $\beta$ {\em flow-homotopy equivalences}.
\end{defn}

\begin{proposition} \label{P:flomotopy-equivalence}
Suppose $X,Y$ are $2$--complexes with semi-flows $\flow^X_s,\flow^Y_s$ and cross sections $\Theta_X \subset X$ and $\Theta_Y \subset Y$.  Further suppose that the first return maps to the cross sections are the restrictions of the time-one maps: $F_X = \flow^X_1|_{\Theta_X} \colon \Theta_X \to \Theta_X$ and  $F_Y = \flow^Y_1|_{\Theta_Y} \colon \Theta_Y \to \Theta_Y$.

If there are maps $\alpha \colon \Theta_X \to \Theta_Y$ and $\beta \colon \Theta_Y \to \Theta_X$ such that
\begin{itemize}
\item $\alpha F_X = F_Y \alpha$ and $\beta F_Y = F_X \beta$, and
\item $\beta \alpha = F_X^k$ and $\alpha \beta = F_Y^k$ for some $k$.
\end{itemize}
then there are flow-homotopy inverse maps $\hat\alpha \colon X \to Y$ and $\hat\beta \colon Y \to X$ extending $\alpha$ and $\beta$, respectively.
\end{proposition} 

\begin{proof}
First, let $M_{F_X}$ be the mapping torus of $F_X \colon \Theta_X \to \Theta_X$ with its suspension semi-flow which we denote $\Psi^X_s$.  Construct maps $h_0^X \colon M_{F_X} \to X$ and $h_1^X \colon X \to M_{F_X}$ by
\[ h_0^X(\theta,t) = \flow^X_t(\theta) \quad\mbox{and}\quad h_1^X(x) = (\flow^X_{\rho_X(x)}(x),1-\rho_X(x)) \]
for $\theta\in \Theta_x$ and $t \in [0,1)$, and where $\rho_X(x) \in (0,1]$ is the return time of $x\in X$ to $\Theta_X$.  That is, $\rho_X(x)$ is the smallest number $t> 0$ so that $\flow^X_t(x) \in \Theta_X$.

\begin{claim}
$h_0^X$ and $h_1^X$ are flow-equivariant, and $h_0^Xh_1^X = \flow^X_1$ and $h_1^X h_0^X = \Psi^X_1$.
\end{claim}

\begin{proof}[Proof of Claim]
This claim follows easily from the definitions, but we spell out a proof here.

First, note that for all $\theta \in \Theta_X$, $t \in [0,1)$ and $s > 0$ we have
\begin{eqnarray*} 
h_0^X(\Psi^X_s(\theta,t)) & = & h_0^X(F_X^{\lfloor s+t \rfloor}(\theta),s+t - \lfloor s+t \rfloor)\\
& = & \flow^X_{s+t - \lfloor s+t \rfloor}(F_X^{\lfloor s+t \rfloor}(\theta))\\
& = & \flow^X_{s+t - \lfloor s+t \rfloor} \flow^X_{\lfloor s+t \rfloor}(\theta)\\
& = & \flow^X_{s+t}(\theta) = \flow^X_s(\flow^X_t(\theta))\\
& = & \flow^X_s(h_0^X(\theta,t)). \end{eqnarray*}
Thus $h_0^X$ is flow-equivariant, as required.

Every $x \in X$ has the form $x = \flow^X_t(\theta)$ for some $\theta \in \Theta_X$ and $0 \leq t < 1$.  Then $\rho_X(x) = 1-t$, and hence
\[ h_1^X(x) = h_1^X(\flow^X_t(\theta)) = (\flow^X_{1-t}\flow^X_t(\theta),1-(1-t))= (F_X(\theta),t). \]
Therefore
\begin{eqnarray*}
h_1^X(\flow^X_s(x)) & = & h_1^X(\flow^X_s\flow^X_t(\theta)) = h_1^X (\flow^X_{s+t}(\theta))\\
& = & h_1^X (\flow^X_{s+t - \lfloor s+t \rfloor}F_X^{\lfloor s+t \rfloor}(\theta))\\
& = & (F_X^{\lfloor s+t \rfloor + 1}(\theta),s+t - \lfloor s+t \rfloor)\\
& = & \Psi^X_{s+t}(F_X(\theta),0) = \Psi^X_s(\Psi^X_t(F_X(\theta),0))\\
&  = & \Psi^X_s(F_X(\theta),t) = \Psi^X_s(h_1^X(x)).
\end{eqnarray*}
Thus $h_1^X$ is also flow-equivariant.

Next let $\theta \in \Theta_X$ and $t \in [0,1)$.  Then $\rho_X(\flow^X_t(\theta)) = 1-t$, and thus 
\begin{eqnarray*}
h_1^Xh_0^X(\theta,t) & = & h_1^X(\flow^X_t(\theta))\\
 &  = & (\flow^X_{1-t}(\flow^X_t(\theta)),1-(1-t)) \\
 & =  & (\flow^X_1(\theta),t) = (F_X(\theta),t)\\
 & = & \Psi^X_1(\theta,t).
 \end{eqnarray*}
On the other hand, for all $x\in X$ we have
\[ h_0^Xh_1^X(x) = h_0^X(\flow^X_{\rho_{X}(x)}(x),1-\rho_X(x)) = \flow^X_{1-\rho_X(x)}(\flow^X_{\rho_X(x)}(x)) = \flow^X_1(x).\]
This completes the proof the claim.
\end{proof}

Next, we note that because $\alpha F_X = F_Y \alpha$ and $\beta F_Y = F_X \alpha$, the maps $\alpha \colon \Theta_X \to \Theta_Y$ and $\beta \colon \Theta_Y \to \Theta_X$ determine flow-equivariant maps between mapping tori
\[ \alpha' \colon M_{F_X} \to M_{F_Y} \mbox{ and } \beta' \colon M_{F_Y} \to M_{F_X} \]
given by
\[ \alpha'(\theta,t) = (\alpha(\theta),t) \mbox{ and } \beta'(\eta, t) = (\beta(\eta),t) \]
for all $\theta \in \Theta_X$, $\eta \in \Theta_Y$ and $0 \leq t < 1$.  Since $\beta \alpha = F_X^k$ and $\alpha \beta = F_Y^k$, we have $\beta' \alpha'(\theta,t) = (F_X^k(\theta),t) = \Psi^X_k(\theta,t)$ and $\alpha' \beta'(\eta,t) = (F_Y^k(\eta),t) = \Psi^Y_k(\eta,t)$.

To complete the proof, we must construct maps
\[ \hat\alpha \colon X \to Y \mbox{ and } \hat\beta \colon Y \to X. \]
These are simply the compositions of the maps above:
\[ \hat\alpha = h_0^Y \alpha' h_1^X \mbox{ and } \hat\beta = h_0^X \beta' h_1^Y \]
where $h_0^Y \colon M_{F_Y} \to Y$ and $h_1^Y \colon Y \to M_{F_Y}$ are defined similar to $h_0^X$ and $h_1^X$, respectively.
As a composition of flow-equivariant maps, these are flow-equivariant.  Finally, using the flow-equivariance and the properties of these maps we obtain
\begin{eqnarray*}
\hat\beta \hat\alpha & = & (h_0^X \beta' h_1^Y)(h_0^Y \alpha' h_1^X) = h_0^X \beta' (h_1^Y h_0^Y) \alpha' h_1^X  = h_0^X \beta' \Psi^Y_1 \alpha' h_1^X\\
& = & h_0^X \Psi^X_1 (\beta' \alpha') h_1^X = h_0^X \Psi^X_1 \Psi^X_k h_1^X\\
& = & h_0^X \Psi^X_{k+1} h_1^X = \flow^X_{k+1} h_0^X h_1^X\\
& = & \flow^X_{k+1} \flow^X_1 = \flow^X_{k+2} \end{eqnarray*}
A similar calculation proves $ \hat\alpha\hat\beta = \flow^Y_{k+2}$.
\end{proof}

\section{A few covering constructions}

The proof of Theorem~\ref{T:other section iwip tech} relies on some constructions of, and facts about, covers of $2$--complexes $Y$ with semi-flows $\flow$.  We will freely use facts from covering space theory, typically without mentioning them explicitly.  To begin, we note that for any cover $p \colon \widetilde Y \to Y$ there is a {\em lifted semi-flow}, $\widetilde \flow$ on $\widetilde Y$.  This lifted semiflow has the property that $p\widetilde \flow_t = \flow_t p$ for all $t \geq 0$.  This semi-flow is obtained by viewing $\flow_t p$ as a homotopy of $p$ and lifting this to the unique homotopy of the identity on $\widetilde Y$.  
Observe that for each covering transformation $T \colon \widetilde Y \to \widetilde Y$, the families $T\widetilde \flow_t$ and $\widetilde \flow_t T$  give two homotopies of $T$ that both lift the homotopy $\flow_t p$ of $p$. By the uniqueness of lifted homotopies, it follows that $T\widetilde\flow_t = \widetilde\flow_t T$. Therefore $\widetilde \flow$ commutes with the group of covering transformations of $\widetilde Y$.

\begin{proposition} \label{P:graph lift to cover}
Suppose $Y$ is a connected $2$--complex with a semiflow $\flow$ and a connected section $\Theta \subset Y$ such that the first return map $f \colon \Theta \to \Theta$ is a homotopy equivalence, and so that the semiflow is parameterized so that the restriction of the time-one map is $f$,  that is, $\flow_1|_\Theta = f$.

Suppose $\Delta \to \Theta$ is a connected finite sheeted covering space and $g \colon \Delta \to \Delta$ is a lift of a positive power $f^n$ of $f$.  Then there is a finite sheeted covering space $p \colon \widetilde Y \to Y$ so that the restriction of $p$ to any one of the components of $p^{-1}(\Theta)$ is isomorphic to $\Delta \to \Theta$, and so that $\Delta$ is a section of the lifted semi-flow with first return map equal to $g$.
\end{proposition}

\begin{proof}  Let $[\Theta] \in H^1(Y;\R)$ be the dual to $\Theta$.  Since $f$ is a homotopy equivalence, $\pi_1(\Theta) = \ker([\Theta]) \triangleleft \pi_1(Y)$, and we let $\widetilde Y_\Delta \to \widetilde Y_\Theta \to Y$ be the covers corresponding to $\pi_1(\Delta) < \pi_1(\Theta) < \pi_1(Y)$.  Write $\flow^\Delta$ and $\flow^\Theta$ for the lifted semi-flows to these covers.

The inclusion of $\Theta$ into $Y$ lifts to an embedding $\Theta = \Theta_0 \subset \widetilde Y_\Theta$ inducing an isomorphism on fundamental groups.  Since $\flow_1$ restricts to the first return map on $\Theta$, $\flow^\Theta_1(\Theta_0) \subset \widetilde Y_\Theta$ is another lift of $\Theta$, differing from $\Theta_0$ by a covering transformation $t$ that generates the infinite cyclic covering group of $\widetilde Y_\Theta \to Y$.  Let $\Theta_n = t^n \Theta_0$, for all $n \in \Z$, so that $\Theta_1 = t \Theta_0 = \flow^\Theta_1(\Theta_0)$.  Then $t^{-1} \flow^\Theta_1|_{\Theta_0}$ is precisely the map $f \colon \Theta \to \Theta$.  Since $\flow^\Theta$ commutes with $t$, we have  $t^{-k} \flow^\Theta_k|_{\Theta_0}=f^k$ for all $k \geq 1$.

There is also an embedding $\Delta = \Delta_0 \subset \widetilde Y_\Delta$ inducing an isomorphism on fundamental groups so that the restriction of $\widetilde Y_\Delta \to \widetilde Y_\Theta$ to $\Delta_0$ is the covering $\Delta \to \Theta$.  Since $t^{-n} \flow^\Theta_n|_{\Theta_0} = f^n$, and since $\pi_1(\Delta_0) \to \pi_1(\widetilde Y_\Delta)$ is an isomorphism, the lift $g \colon \Delta \to \Delta$ of $f^n$ can be extended to a lift $\widetilde Y_\Delta \to \widetilde Y_\Delta$ of $t^{-n} \flow^\Theta_n$.  On the other hand, $t^{-n} \flow^\Theta_n = \flow^\Theta_n t^{-n}$ is homotopic via the semi-flow $\flow^\Theta$ to $t^{-n}$.  The lifted semiflow is the lift of the homotopy, and it follows that we can lift ($t^{-n}$ and hence) $t^n$ to a map $T \colon \widetilde Y_\Delta \to \widetilde Y_\Delta$ so that $T^{-1} \flow^\Delta_n \colon \widetilde Y_\Delta \to \widetilde Y_\Delta$ is the chosen lift of $t^{-n} \flow^\Theta_n$.  

Being a lift of a covering map, $T$ is itself a covering transformation of $\widetilde Y_\Delta \to Y$, and we form the quotient $\widetilde Y = \widetilde Y_\Delta/\langle T \rangle$.  Since $T$ commutes with $\flow^\Delta$, it descends to a semi-flow $\widetilde \flow$ on $\widetilde Y$.
The restriction to $\Delta_0$ of $\widetilde Y_\Delta \to \widetilde Y$ is an embedding of $\Delta$ into $\widetilde Y$, and the first return of $\widetilde \flow$ to this copy of $\Delta$ occurs precisely at time $n$.
Since we have factored out by $\langle T \rangle$, this first return map is the descent of $T^{-1} \flow^\Delta_n$ restricted to $\Delta_0$, and is thus precisely $g$, as required.
\end{proof}

The following provides a converse to the previous proposition which we will need.

\begin{proposition} \label{P:cover to graph lift}
Suppose that $Y$ is a connected $2$--complex with a semiflow $\flow$ and connected cross section $\Theta \subset Y$ so that the first return map $f \colon \Theta \to \Theta$ is the restriction of the time-$1$ map, $\flow_1|_\Theta = f$ and is a homotopy equivalence.  Given a connected, finite sheeted covering space $p \colon \widetilde Y \to Y$, any component $\Delta \subset p^{-1}(\Theta)$ is a section, and the first return map $g \colon \Delta \to \Delta$ of the lifted semi-flow is a lift of a power of $f$.
\end{proposition}
\begin{proof}
Every cover of $Y$ is a quotient of the universal covering $\widehat Y \to Y$, and the proposition will follow easily from a good description of this $\widehat Y$, which we now explain.  We first let $\widetilde Y_\Theta \to Y$ denote the cover corresponding to $\pi_1(\Theta) = \ker([\Theta])$.  As in the previous proof, we have homeomorphic copies of $\Theta$ in $\widetilde Y_\Theta$, which we denote $\{ \Theta_n\}_{n \in \Z}$, so that a generator $t$ of the covering group has $t \Theta_n = \Theta_{n+1}$ for all $n$.   Furthermore, the lifted semi-flow $\flow^\Theta$ to $\widetilde Y_\Theta$ has $\flow^\Theta_1(\Theta_n) = \Theta_{n+1}$, and $t^{-1} \flow^\Theta_1 \colon \Theta_n \to \Theta_n$ is the map $f$, with respect to the homeomorphism $\Theta_n \cong \Theta$ obtained by restricting $\widetilde Y_\Theta \to Y$ to $\Theta_n$.

Since the inclusion $\Theta_n \subset \widetilde Y_\Theta$ is an isomorphism on fundamental group, the universal cover $\widehat Y \to \widetilde Y_\Theta$ contains copies of the universal cover of $\Theta$, say $\{\widehat \Theta_n\}_{n \in \Z}$ so that for each $n$, $\widehat \Theta_n$ is the preimage of $\Theta_n$.  The lifted semiflow $\widehat \flow$ to $\widehat Y$ has time--$1$ map sending $\widehat \Theta_n$ to $\widehat \Theta_{n+1}$ for all $n$.  In particular, for any integer $k > 0$, $\widehat \flow_k (\widehat \Theta_n)= \widehat \Theta_{n+k}$, and $\widehat \flow_k|_{\widehat \Theta_n}$ is a lift of the $k^{th}$ power of $f$ from the $n^{th}$ copy of the universal cover of $\Theta$ to the $(n+k)^{th}$ copy.

Any connected, finite sheeted cover $\widetilde Y \to Y$ is a quotient of $\widehat Y$, the lifted semi-flow $\widetilde \flow$ is the descent of $\widehat \flow$ to $\widetilde Y$, and $\{ \widehat \Theta_n\}_{n \in \Z}$ push down to finitely many graphs $\widetilde \Theta_1,\ldots,\widetilde \Theta_j$ in $\widetilde Y$, each of which is a finite sheeted covering space of $\Theta$ (here $j$ is the subgroup index in $\Z$ of the image of $\pi_1(\widetilde \Theta)$ under the homomorphism $[\Theta]$).  We may choose our indices $1,\ldots,j$ so that $\widehat \Theta_n$ pushes down to $\widetilde \Theta_k$, where $k \equiv n$ mod $j$ for all $n$.  From the description of $\widehat \flow$, it follows that $\widetilde \flow_1(\widetilde \Theta_k) = \widetilde \Theta_{k+1}$, with indices taken modulo $j$.  Consequently, $\widetilde \flow_j(\widetilde \Theta_k) = \widetilde \Theta_k$ for all $k$, and the restriction of $\widetilde \flow_j$ to any one is a lift of $f^j$.  As this is the first return map, we are done.
\end{proof}

\begin{proposition} \label{P:flomotopy covers}
Suppose $X$ and $Y$ are connected $2$--complexes equipped with semi-flows $\flow^X$ and $\flow^Y$, respectively.  Given flow-homotopy inverse maps $\alpha \colon X \to Y$ and $\beta \colon Y \to X$, and a connected finite sheeted cover $p \colon \widetilde X \to X$, there exists a connected finite sheeted cover $q \colon \widetilde Y \to Y$ and lifts of $\alpha$ and $\beta$ which are flow-homotopy inverses:
\[ \xymatrix{  \widetilde X \ar@/^/^{\widetilde \alpha}[r] \ar_p[d] & \widetilde Y \ar@/^/^{\widetilde \beta}[l] \ar^q[d] \\
X \ar@/^/^\alpha[r] & Y \ar@/^/^\beta[l] }\]
\end{proposition}
\begin{proof}
Let $q \colon \widetilde Y \to Y$ be the connected cover corresponding to $\alpha_*(p_*(\pi_1(\widetilde X)))$.  Since $p_*(\pi_1(\widetilde X))$ has finite index in $\pi_1(X)$, and $\alpha_*$ is an isomorphism, it follows that $q_*(\pi_1(\widetilde Y))$ has finite index in $\pi_1(Y)$, and hence $q$ is a finite sheeted cover.

From basic covering space theory, $\alpha$ lifts to a map $\widetilde \alpha \colon \widetilde X \to \widetilde Y$ so that $q \widetilde \alpha = \alpha p$.  Since $\beta$ is a homotopy inverse of $\alpha$, $\beta_*(q_*(\pi_1(\widetilde Y)))$ is (conjugate to) $p_*(\pi_1(\widetilde X))$. By changing the basepoint of $\widetilde{Y}$ to adjust this conjugate, it follows that there is a lift $\widetilde \beta \colon \widetilde Y \to \widetilde X$ so that $p \widetilde \beta = \beta q$.  Let $\widetilde \flow^X$ and $\widetilde \flow^Y$ denote the lifted semi-flows, and note that $p \widetilde \beta \widetilde \alpha = \beta q \widetilde \alpha = \beta \alpha p = \flow^X_K p$ for some $K > 0$.   Therefore, $\widetilde \beta \widetilde \alpha$ is a lift of $\flow^X_K$.  

Since $\flow^X_t$, $t \in [0,K]$ defines a homotopy from the identity to $\flow^X_K$, we can lift the homotopy and thus $\widetilde \beta \widetilde \alpha$ is homotopic (via some lift of $\flow^X_t$) to a map covering the identity, i.e.~a covering transformation for $p$.
Composing $\widetilde \beta$ with the inverse of this covering transformation, we get another lift of $\beta$ (which we continue to call $\widetilde \beta$) so that now $\widetilde \beta \widetilde \alpha = \widetilde \flow^X_K$.  We claim that $\widetilde \alpha$ and $\widetilde \beta$ are flow-homotopy inverses.

First, we verify that $\widetilde \alpha$ and $\widetilde \beta$ are flow-equivariant.  To see this, first note that for every $\widetilde x \in \widetilde X$, the paths $t \mapsto \widetilde \alpha \widetilde \flow^X_t(\widetilde x)$ and $t \mapsto \widetilde \flow^Y_t \widetilde \alpha(\widetilde x)$ are both lifts of the path $t \mapsto \alpha \flow^X_t p(\widetilde x) = \flow^Y_t \alpha p(\widetilde x)$.   Since these have the same value $\widetilde \alpha(\widetilde x)$ at time $t = 0$, uniqueness of path lifting guarantees that $\widetilde \flow^Y_t  \widetilde \alpha = \widetilde \alpha \widetilde \flow^X_t$, so $\widetilde \alpha$ is flow-equivariant.  The same argument works for $\widetilde \beta$.

Our choice of $\widetilde \beta$ ensures that $\widetilde \beta \widetilde \alpha = \widetilde \flow^X_K$.  A similar calculation as above ensures $\widetilde \alpha \widetilde \beta$ differs from $\widetilde \flow^Y_K$ by a covering transformation.  To complete the proof, we must show that this covering transformation is trivial.  To do this, we pick any point in the image of $\widetilde \alpha$, $\widetilde \alpha(\widetilde x) \in \widetilde Y$, and then observe that
\[ \widetilde \alpha \widetilde \beta \widetilde \alpha(\widetilde x ) = \widetilde \alpha \widetilde \flow^X_K(\widetilde x) = \widetilde \flow^Y_K \widetilde \alpha(\widetilde x). \]
Thus $\widetilde \alpha \widetilde \beta$ agrees with $\widetilde \flow^Y_K$ at the point $\widetilde \alpha(\widetilde x)$.  But since these differ by a covering transformation and they agree at a point, it follows that the covering transformation is the identity, and hence $\widetilde \alpha \widetilde \beta = \widetilde \flow^Y_K$.
\end{proof}

\section{Full irreducibility}

\begin{proof}[Proof of Theorem~\ref{T:other section iwip tech}]

Recall that we have the folded mapping torus $X = X_f$, for $f \colon \Gamma \to \Gamma$ an expanding irreducible train track representative of a hyperbolic fully irreducible automorphism.  We have $\Theta \subset X$ a section with first return map $f_\Theta \colon \Theta \to \Theta$, an expanding irreducible train track, inducing an automorphism on the stable quotient.  This automorphism is represented by the expanding irreducible train track map $\bar f_\Theta \colon \bar \Theta \to \bar \Theta$ from Theorem~\ref{T:promoting train track maps}.   Now suppose that $(\bar f_{\Theta})_*$ is not fully irreducible.

\begin{claim}  There is a finite sheeted covering $\Delta \to \bar \Theta$, a lift $g \colon \Delta \to \Delta$ of a power of $\bar f_\Theta$, and a proper subgraph $\Omega \subset \Delta$ containing at least one edge so that $g(\Omega) = \Omega$.
\end{claim}
\begin{proof}
Since $(\bar f_{\Theta})_*$ is not fully irreducible, there exists $n > 0$ for which we may choose a vertex $z\in \bar{\Theta}$ with $\bar{f}^n_{\Theta}(z) = z$ and free factor $H$ of $\pi_1(\bar\Theta) = \pi_1(\bar\Theta,z)$ such that $(\bar f^n_\Theta)_*(H)$ is conjugate to $H$.

Let $p\colon (\widetilde \Omega,\tilde z) \to (\bar \Theta,z)$ denote the cover corresponding to $H$ and choose a vertex $\tilde{z}'\in p\inv(z)$ so that $p_*(\pi_1(\widetilde \Omega,\tilde{z}')) = (\bar f^n_{\Theta})_*(H)$.  Basic covering space theory guarantees that there is a unique lift $h \colon \widetilde \Omega \to \widetilde \Omega$ of $\bar f_\Theta^n$ sending $\tilde{z}$ to $\tilde{z}'$.  Let $\gamma \colon S^1 \to \widetilde \Omega$ be any non-null-homotopic closed curve. Since $(\bar f_{\Theta})_*$ is hyperbolic, the sequence of curves $h^k \circ \gamma$ is an infinite sequence of distinct homotopy classes. Tightening each curve in the sequence gives an infinite sequence of curves $\{\gamma_k\}$ in the Stallings core $\bar \Omega \subset \widetilde \Omega$, each without backtracking, representing distinct homotopy classes.  Furthermore, since neither $h$ nor tightening can increase the number of illegal turns in a loop of $\bar\Omega$, the number of illegal turns of $\gamma_k$ is uniformly bounded as $k\to \infty$. 
It follows that the length of the maximal legal segment of $\gamma_k$ must tend to infinity with $k$.  From a sufficiently long legal segment we can construct a legal loop $\delta$ contained in $\bar \Omega$.  The loops $h^k \circ \delta$ must be legal for all $k > 0$, and hence must be contained in the core of $\bar \Omega$.  It follows that
\[ \Omega = \bigcap_{k > 0} \bigcup_{j \geq k} h^j(\delta(S^1)) \subset \bar \Omega \]
is a nonempty subgraph of $\bar \Omega$ with at least one edge, and that $h(\Omega) = \Omega$ since
\[ h(\Omega) = h\left(\bigcap_{k > 0} \bigcup_{j \geq k} h^j(\delta(S^1))\right) = \bigcap_{k > 0} \bigcup_{j \geq k} h^{j+1}(\delta(S^1)) = \Omega.\]

Next let $\ell > 0$ be such that $h^\ell$ has a fixed vertex $w \in \Omega$.  Thus $h^\ell$ is a lift of $\bar f_{\Theta}^{n\ell}$, and $\bar f_{\Theta}^{n \ell}$ fixes the image $v \in \bar \Theta$ of $w$.  By Hall's Theorem (i.e.~separability of finitely generated subgroups of free groups), there are covering maps
\[ \widetilde \Omega \to \Delta \to \bar \Theta \]
such that $\Delta \to \bar \Theta$ is a finite sheeted covering, and so that $\widetilde \Omega \to \Delta$ restricts to an embedding on $\Omega$.  We use this fact to identify $\Omega$ and the point $w$ with their images in $\Delta$, noting that $\Omega \subset \Delta$ is a proper subgraph containing at least one edge.

Finally, choose a power $\bar f_{\Theta}^{jn\ell}$ such that $(\bar f_{\Theta}^{jn\ell})_*$ fixes the image of $\pi_1(\Delta,w)$ in $\pi_1(\bar \Theta,v)$.  By covering space theory again, we may choose a lift $g \colon \Delta \to \Delta$ of $\bar f_{\Theta}^{jn\ell}$ fixing the image of $w$ in $\Delta$.  It follows that the restriction of $g$ to $\Omega$ agrees with the restriction of $h^{j\ell}$ to $\Omega$ (via the identification from the covering $\widetilde \Omega \to \Delta$).  In particular, $g(\Omega) =\Omega$.
\end{proof}

As in \cite{DKL2}, we may reparameterize the semi-flow on $X$ so that the first return map to $\Theta$ is the time-one map.  Applying Proposition~\ref{P:flomotopy-equivalence} (to the maps $\bar\Theta\to \Theta$ and $\Theta\to\bar\Theta$ provided by Theorem~\ref{T:promoting train track maps}), we get flow-homotopy inverse maps $\alpha$ and $\beta$ between the mapping torus $M_{\bar f_{\Theta}}$ and $X$.  Note that these maps restrict to graph maps between $\bar \Theta$ and $\Theta$.

Let $\Delta \to \bar \Theta$ be the finite sheeted cover, $g \colon \Delta \to \Delta$ the lift of a power of $\bar f_{\Theta}$, and $\Omega \subset \Delta$ the proper subgraph with at least one edge and $g(\Omega) = \Omega$, all from the claim.  By Proposition~\ref{P:graph lift to cover}, there is a cover $p \colon \widetilde M_{\bar f_{\Theta}} \to M_{\bar f_{\Theta}}$ so that $p$ restricted to a component of $p^{-1}(\bar \Theta)$ is isomorphic to $\Delta \to \bar \Theta$.  Proposition~\ref{P:flomotopy covers} then provides flow-homotopy inverse lifted maps to a cover $\widetilde X$ of $X$, denoted $\widetilde \alpha$ and $\widetilde \beta$.  Letting $\widetilde \Gamma$ denote a component of the preimage of $\Gamma$ in $\widetilde X$, we have the following diagram:

\[ \xymatrix{ \Delta \ar[r] \ar[d] & \widetilde M_{\bar f_{\Theta}} \ar@/^/^{\widetilde \alpha}[r] \ar[d] & \widetilde X \ar@/^/^{\widetilde \beta}[l] \ar[d] & \widetilde \Gamma \ar[l] \ar[d] \\
\bar \Theta \ar[r] & M_{\bar f_{\Theta}} \ar@/^/^\alpha[r] & X \ar@/^/^\beta[l] &  \Gamma \ar[l] }\]
Let $\Psi$ and $\psi$ denote the flows on $\widetilde M_{\bar f_{\Theta}}$ and $\widetilde X$, respectively, and let $K > 0$ be so that $\widetilde \beta \widetilde \alpha = \Psi_K$ and $\widetilde\alpha \widetilde\beta = \psi_K$.  Note that this implies $\widetilde \alpha$ and $\widetilde \beta$ are surjective, since $\Psi_K$ and $\psi_K$ are.

There is a proper, flow invariant subset $Z_\Omega \subset \widetilde M_{\bar f_{\Theta}}$ defined by
\[ Z_\Omega = \bigcup_{t \geq 0} \Psi_t (\Omega).\]
Since the first return of $\Psi$ to $\Delta$ is $g$, which is surjective, and since $g(\Omega) = \Omega$, it follows that $\Psi_t(Z_\Omega) = Z_\Omega$ and $\Psi_t(\widetilde M_{\bar f_{\Theta}}) \neq Z_\Omega$, for every $t \geq 0$.
Now flow equivariance implies
\[ \psi_t(\widetilde \alpha(Z_\Omega)) =  \widetilde \alpha(\Psi_t(Z_\Omega)) = \widetilde \alpha(Z_\Omega) \]
Furthermore, suppose that $\psi_t(\widetilde X) = \widetilde \alpha(Z_\Omega)$ for some $t$.  Then surjectivity and equivariance of $\widetilde \beta$ implies
\[ Z_\Omega = \Psi_K(Z_\Omega) = \widetilde \beta(\widetilde \alpha(Z_\Omega)) = \widetilde \beta(\psi_t(\widetilde X)) = \Psi_t(\widetilde \beta(\widetilde X)) = \Psi_t(\widetilde M_{\bar f_{\Theta}}) \neq Z_\Omega , \]
a contradiction.  Therefore, $\psi_t(\widetilde X) \neq \widetilde \alpha(Z_\Omega)$ for all $t \geq 0$.

Since $\alpha$ sends edges of $\bar \Theta$ to edges of $\Theta$, we see that $\widetilde \alpha$ sends edges of $\Delta$ to edges of the preimage of $\Theta$ in $\widetilde X$.  It follows that $\widetilde \alpha(Z_\Omega)$ contains an open subset of a $2$--cell of $\widetilde X$ and thus that  $\widetilde \alpha(Z_\Omega)$ eventually flows over an entire edge $e$ of the component $\widetilde \Gamma$ of the preimage of $\Gamma$.
Now we note that the first return map to $\widetilde \Gamma$ is a lift of a power of $f$ by Proposition~\ref{P:cover to graph lift}.  Denote this first return map $r: \widetilde \Gamma\to \widetilde \Gamma$. Thus $r$ is a train track map. 

A result of Bestvina-Feighn-Handel~\cite[Proposition 2.4]{BFH97}, or alternatively, a recent result of Dowdall and Taylor~\cite[Proposition~5.1]{DT16} imply that if a hyperbolic fully irreducible automorphism of $F_N$ preserves a subgroup of finite index in $F_N$, then the restriction of the automorphism to that subgroup induces a fully irreducible automorphism of the subgroup. 

Therefore $r$ induces a fully irreducible automorphism of $\pi_1(\widetilde \Gamma)$, and being a train track representative of that automorphism, $r$ is an expanding irreducible train track map (e.g. by \cite[Lemma 2.4]{K14}). 
In particular, the edge $e$ must eventually map over the entire graph $\widetilde \Gamma$ by some power of the first return map $r$.   It follows that the $\flow$--invariant subset $\widetilde \alpha(Z_\Omega)$ contains $\widetilde \Gamma$.  But since $\widetilde \Gamma$ is a section of $\psi$, this implies that $\widetilde \alpha(Z_\Omega) = \widetilde X$, which is a contradiction.
\end{proof}

\bibliography{fbc3}{}

\begin{thebibliography}{GMSW}

\bibitem[AKHR]{AHR}
Yael Algom-Kfir, Eriko Hironaka, and Kasra Rafi.
\newblock Digraphs and cycle polynomials for free-by-cyclic groups.
\newblock {\em Geom. Topol.}, 19(2):1111--1154, 2015.

\bibitem[AKR]{AR}
Yael Algom-Kfir and Kasra Rafi.
\newblock Mapping tori of small dilatation expanding train-track maps.
\newblock {\em Topology Appl.}, 180:44--63, 2015.

\bibitem[BFH]{BFH97}
M.~Bestvina, M.~Feighn, and M.~Handel.
\newblock Laminations, trees, and irreducible automorphisms of free groups.
\newblock {\em Geom. Funct. Anal.}, 7(2):215--244, 1997.

\bibitem[BG]{BG04}
R.~Bieri and R.~Geoghegan.
\newblock Controlled topology and group actions.
\newblock In {\em Groups: topological, combinatorial and arithmetic aspects},
  volume 311 of {\em London Math. Soc. Lecture Note Ser.}, pages 43--63.
  Cambridge Univ. Press, Cambridge, 2004.

\bibitem[BH]{BH92}
Mladen Bestvina and Michael Handel.
\newblock Train tracks and automorphisms of free groups.
\newblock {\em Ann. of Math. (2)}, 135(1):1--51, 1992.

\bibitem[BNS]{BNS}
Robert Bieri, Walter~D. Neumann, and Ralph Strebel.
\newblock A geometric invariant of discrete groups.
\newblock {\em Invent. Math.}, 90(3):451--477, 1987.

\bibitem[Bro]{Brown}
Kenneth~S. Brown.
\newblock Trees, valuations, and the {B}ieri-{N}eumann-{S}trebel invariant.
\newblock {\em Invent. Math.}, 90(3):479--504, 1987.

\bibitem[CL]{CL}
Christopher~H. Cashen and Gilbert Levitt.
\newblock Mapping tori of free group automorphisms, and the
  {B}ieri-{N}eumann-{S}trebel invariant of graphs of groups.
\newblock {\em J. Group Theory}, 19(2):191--216, 2016.

\bibitem[DKL1]{DKL2}
Spencer Dowdall, Ilya Kapovich, and Christopher~J. Leininger.
\newblock {M}c{M}ullen polynomials and {L}ipschitz flows for free-by-cyclic
  groups.
\newblock 2013.
\newblock To appear in {\it JEMS}. preprint arXiv:1310.7481.

\bibitem[DKL2]{DKL}
Spencer Dowdall, Ilya Kapovich, and Christopher~J. Leininger.
\newblock Dynamics on free-by-cyclic groups.
\newblock {\em Geom. Topol.}, 19(5):2801--2899, 2015.

\bibitem[DT]{DT16}
Spencer Dowdall and Samuel~J. Taylor.
\newblock The co-surface graph and the geometry of hyperbolic free group
  extensions.
\newblock {\em J. Topol.}, 10(2):447--482, 2017.

\bibitem[DV]{DV96}
Warren Dicks and Enric Ventura.
\newblock {\em The group fixed by a family of injective endomorphisms of a free
  group}, volume 195 of {\em Contemporary Mathematics}.
\newblock American Mathematical Society, Providence, RI, 1996.

\bibitem[Gau1]{Gau1}
Fran{\c{c}}ois Gautero.
\newblock Dynamical 2-complexes.
\newblock {\em Geom. Dedicata}, 88(1-3):283--319, 2001.

\bibitem[Gau2]{Gau3}
Fran{\c{c}}ois Gautero.
\newblock Cross sections to semi-flows on 2-complexes.
\newblock {\em Ergodic Theory Dynam. Systems}, 23(1):143--174, 2003.

\bibitem[GMSW]{GMSW}
Ross Geoghegan, Michael~L. Mihalik, Mark Sapir, and Daniel~T. Wise.
\newblock Ascending {HNN} extensions of finitely generated free groups are
  {H}opfian.
\newblock {\em Bull. London Math. Soc.}, 33(3):292--298, 2001.

\bibitem[Kap]{K14}
Ilya Kapovich.
\newblock Algorithmic detectability of iwip automorphisms.
\newblock {\em Bull. Lond. Math. Soc.}, 46(2):279--290, 2014.

\bibitem[Lev]{Levitt}
Gilbert Levitt.
\newblock {${\bf R}$}-trees and the {B}ieri-{N}eumann-{S}trebel invariant.
\newblock {\em Publ. Mat.}, 38(1):195--202, 1994.

\bibitem[Rey]{Reynolds}
Patrick Reynolds.
\newblock Dynamics of irreducible endomorphisms of {${F}_n$}.
\newblock 2010.
\newblock preprint arXiv:1008.3659.

\bibitem[Thu]{ThuN}
William~P. Thurston.
\newblock A norm for the homology of {$3$}-manifolds.
\newblock {\em Mem. Amer. Math. Soc.}, 59(339):i--vi and 99--130, 1986.

\bibitem[Wan]{Wang}
Zhifeng Wang.
\newblock {\em Mapping tori of outer automorphisms of free groups}.
\newblock PhD thesis, Rutgers University, May 2002.

\end{thebibliography}
\bibliographystyle{alphanum}

\end{document}